\documentclass[12pt,reqno]{amsproc}%
\usepackage{amsfonts}
\usepackage{amsmath}
\usepackage{amssymb}
\usepackage{graphicx}%
 \theoremstyle{plain}
\newtheorem{theorem}{Theorem}[section]

\newtheorem{corollary}[theorem]{Corollary}

\newtheorem{lemma}[theorem]{Lemma}

\newtheorem{proposition}[theorem]{Proposition}
\newtheorem{remark}[theorem]{Remark}

\numberwithin{equation}  {section}

\begin{document}

\title{Similarity degree of type II$_{1}$ von Neumann algebras with Property $\Gamma$}

\author{Wenhua Qian}
\address{Wenhua Qian \\
        Demartment of Mathematics and Statistics \\
         University of New Hampshire\\
         Durham, NH 03824;   Email: wuf2@wildcats.unh.edu}
\author{Don  Hadwin}
\address{Don  Hadwin \\
        Demartment of Mathematics and Statistics \\
         University of New Hampshire\\
         Durham, NH 03824;   Email: don@unh.edu}
\author{Junhao Shen}
\address{Junhao Shen \\
        Demartment of Mathematics and Statistics \\
         University of New Hampshire\\
         Durham, NH 03824;   Email:  Junhao.Shen@unh.edu}
%\thanks{Partially supported by a grant from the National Science Foundation}

\begin{abstract}
In this paper, we discuss some equivalent definitions of Property
$\Gamma$ for a type II$_1$ von Neumann algebra. Using these equivalent definitions, we prove that the Pisier's similarity degree of a type II$_1$
von Neumann algebra with Property $\Gamma$ is equal to $3$.

\end{abstract}

% Activate to display a given date or no date

%\newtheorem{theorem}{Theorem}[section]
%\newtheorem{conclusion}{Conclusion}[section]
%\newtheorem{corollary}{Corollary}[section]
%\newtheorem{definition}{Definition}[section]
%\newtheorem{example}{Example}[section]
%\newtheorem{lemma}{Lemma}[section]
%\newtheorem{notation}{Notation}[section]
%\newtheorem{proposition}{Proposition}[section]
%\newtheorem{remark}{Remark}[section]
%\newtheorem{solution}{Solution}[section]
%\newtheorem{summary}{Summary}[section]

\maketitle

\section{Introduction}

Kadison's Similarity Problem for a C$^*$-algebra is a longstanding
open problem, which
  asks whether every
    bounded  representation $\rho$ of  a C$^*$-algebra $\mathcal A$ on a Hilbert
space $  H$ is similar to a $*$-representation. i.e. whether there
exists an invertible operator $T$ in $B(H)$, such that
$T\rho(\cdot)T^{-1}$ is a $*$-representation of $\mathcal A$.
Significant progress  toward this famous open problem was obtained in
 \cite {C3} and \cite {H1}.
 We will refer to Pisier's book \cite {Pi5} for a wonderful introduction to the problem and
many of its recent developments.

Similarity degree for a unital C$^*$-algebra $\mathcal A$, denoted
by $d(\mathcal A)$, was defined by Pisier in \cite{Pi2}. Since its
introduction, this new concept has greatly influenced the study of
Kadison's Similarity Problem for C$^*$-algebras. In fact, it was
shown in \cite{Pi2} that Kadison's Similarity Problem for a unital
C$^*$-algebra $\mathcal{A}$ has an affirmative answer if and only if
$d(\mathcal{A})<\infty$. One of the most surprising results on
similarity degree was also obtained by Pisier in \cite{Pi3} when he
proved that, for an infinite dimensional unital C$^*$-algebra
$\mathcal A$, the similarity degree of $\mathcal A$ is equal to $2$
if and only if $\mathcal A$ is a nuclear C$^*$-algebra.

Several results on similarity degree for a unital C$^*$-algebra
have now been known. For example, if there is no tracial state on a unital C$^*$-algebra
$\mathcal{A}$, then $d(\mathcal{A})=3$  (\cite{H1}, \cite{Pi4}). The similarity degree of  a type II$_{1}$ factor $\mathcal M$ with Property $\Gamma$ is less than or equal to $ 5$ (\cite{Pi4}).  This result was later
improved in \cite{C2} to that the similarity degree of  such $\mathcal M$ is equal  to $3$. When $\mathcal{A}$ is
a minimal tensor product of two C*-algebras, one of which is nuclear
and contains matrices of any order,  it was proved in \cite{Po} that  $d(\mathcal{A}) \le
5$. Recently, it was shown in \cite{JW} that, if $\mathcal{A}$ is $\mathcal{Z}$-stable, then
$d(\mathcal{A}) \leq 5$.  In \cite {QS2}, it was shown
that, if  a separable
C*-algebra $\mathcal{A}$ has Property c$^*$-$\Gamma$, then $d(\mathcal{A})
=3$, which implies that a nonnuclear separable $\mathcal{Z}$-stable C$^*$-algebra has similarity degree $3$.

In this paper, we will discuss properties of type II$_{1}$ von Neumann algebras
with Property $\Gamma$ and compute   similarity degree for this class of von Neumann algebras. First result we obtained in the paper is the following characterization of type
II$_1$ von Neumann algebras with Property $\Gamma$.

\

{Theorem \ref{3.3}.} \ {\em Suppose $\mathcal{M}$ is a type II$_{1}$ von Neumann algebra and
 $\mathcal{Z}_{\mathcal{M}}$ is the center of $\mathcal{M}$.
Let $\tau$ be the center-valued  trace on $\mathcal{M}$ such that
$\tau(a)=a$ for any $a \in \mathcal{Z}_{\mathcal{M}}$. Then the
following statements are equivalent.
\begin{enumerate}
\item[(i)] $\mathcal{M}$ has Property $\Gamma$ in the sense of Definition 3.1 in \cite{QS}.
\item[(ii)] There exists a family of nonzero orthogonal central projections $\{ q_{\alpha}: \alpha \in \Omega \}$ in $\mathcal M$ with sum $I$ such that $q_{\alpha} \mathcal{M}$ is a countably decomposable type II$_{1}$ von Neumann algbera with Property $\Gamma$ for each $\alpha \in
\Omega$.
\item[(iii)] For any $n \in \mathbb{N}$, any $\epsilon >0$ and $a_{1}, a_{2}, \dots, a_{k} \in \mathcal{M}$, there exist $n$ orthogonal equivalent projections $p_{1}, p_{2}, \dots, p_{n}$ in $\mathcal{M}$ with sum $I$ such that
$$\tau((p_{i}a_{j}-a_{j}p_{i})^*(p_{i}a_{j}-a_{j}p_{i})) < \epsilon I, \qquad \forall \quad 1 \le i \le n, 1 \le j \le k.$$
\item[(iv)] There exists a positive integer $n_{0} \ge 2$ satisfying {\em for any $\epsilon >0$ and $a_{1}, a_{2}, \dots, a_{k} \in \mathcal{M}$, there exist $n_{0}$ orthogonal equivalent projections $p_{1}, p_{2}, \dots, p_{n_{0}}$ in $\mathcal{M}$ with sum $I$ satisfying
$$\tau((p_{i}a_{j}-a_{j}p_{i})^*(p_{i}a_{j}-a_{j}p_{i})) < \epsilon I, \qquad \forall \quad 1 \le i \le n_{0}, 1 \le j \le
k.$$}
\item[(v)] For any $\epsilon >0$ and $a_{1}, a_{2}, \dots, a_{k} \in \mathcal{M}$, there exists a unitary $u$ in $\mathcal{M}$ such that
\begin{enumerate}
\item[(a)] $\tau(u)=0$;
\item[(b)] $\tau((ua_{j}-a_{j}u)^*(ua_{j}-a_{j}u)) < \epsilon I, \qquad \forall \quad 1 \le j \le k.$
\end{enumerate}
\item[(vi)] For  any $n \in \mathbb{N}$, any normal tracial state $\rho$ on $\mathcal{M}$, any $\epsilon > 0$ and $a_{1}, a_{2}, \dots, a_{k} \in \mathcal{M}$, there exist $n$ orthogonal equivalent projections $p_{1}, p_{2}, \dots, p_{n}$ in $\mathcal{M}$ with sum $I$ such that
$$\Vert p_{i}a_{j}-a_{j}p_{i} \Vert_{2, \rho} < \epsilon, \qquad \forall \quad 1\le i \le n, 1 \le j \le k,$$
where $\Vert \cdot \Vert_{2, \rho}$ is the $2$-norm on $\mathcal{M}$ induced by $\rho$.
\item[(vii)] There exists a positive integer $n_{0} \ge 2$ satisfying {\em for any  normal tracial state $\rho$ on $\mathcal{M}$, any $\epsilon > 0$ and $a_{1}, a_{2}, \dots, a_{k} \in \mathcal{M}$, there exist $n_{0}$ orthogonal equivalent projections $p_{1}, p_{2}, \dots, p_{n_{0}}$ in $\mathcal{M}$ with sum $I$ satisfying
$$\Vert p_{i}a_{j}-a_{j}p_{i} \Vert_{2, \rho} < \epsilon, \qquad \forall \quad 1\le i \le n_{0}, 1 \le j \le
k,$$}
where $\Vert \cdot \Vert_{2, \rho}$ is the $2$-norm on $\mathcal{M}$ induced by $\rho$.
\item[(viii)] For any normal tracial state $\rho$ on $\mathcal{M}$, any $\epsilon >0$ and $a_{1}, a_{2}, \dots, a_{k} \in \mathcal{M}$, there exists a unitary $u$ in $\mathcal{M}$ such that (a)  $\tau(u)=0$; and
 (b)  $\Vert ua_{j}-a_{j}u \Vert_{2, \rho} < \epsilon,$  for all $ 1 \le j \le k,$
where $\Vert \cdot \Vert_{2, \rho}$ is the $2$-norm on $\mathcal{M}$ induced by $\rho$.
\end{enumerate}}
\

We should remark that, when  a type II$_1$ von Neumann algebra $\mathcal M$ is in fact a type II$_1$ factor, our definition of Property $\Gamma$ coincides with Murray and von Neumann's original definition in \cite{Mv}.

Combining preceding {Theorem \ref{3.3} and results in \cite{QS2}, we are able to calculate the exact value  of similarity degree for a type II$_1$ von Neumann algebra with Property $\Gamma$ and obtain the next result as a generalization of Christensen's result in \cite{C2}.
\vspace{0.2cm}

{Theorem \ref{4.2}.} \ {\em If $\mathcal{M}$ is a type II$_{1}$ von Neumann algebra with Property $\Gamma$, then the similarity degree $d(\mathcal{M}) =3$. }
\

Next we apply  Theorem \ref{4.2} to calculate values of similarity degrees for two classes of C$^*$-algebras, which were also considered by Pisier in \cite{Pi4}.

Suppose $\mathcal A$ is unital C$^*$-algebra. Let ${\mathcal I}$ be some index set and
$$l_{\infty} ({\mathcal I}, \mathcal{A}) = \{ (x_{i})_{i \in \mathcal{I}}: \text{ for each } i \in {\mathcal I}, x_{i} \in \mathcal{A} \text{ and } \sup\limits_{i \in {\mathcal I}} \Vert x_{i} \Vert < \infty \}.$$

\

{Corollary} \ref{4.4}.
{\em If $\mathcal{M}$ is a type II$_{1}$ factor with Property $\Gamma$, then $d( l_{\infty} ({\mathcal I}, \mathcal{M}) ) = 3$ for any index set $\mathcal{I}$.
 }
 \vspace{0.2cm}

{Corollary} \ref{4.5}.
{\em Let $C=M_{2}(\mathbb{C}) \otimes M_{2}(\mathbb{C}) \otimes \dots$
(infinite C$^*$-tensor product of $2 \times 2$ matrix algebras).
Then, for any infinite index set ${\mathcal I}$, $d(l_{\infty}
({\mathcal I}, C) ) = 3$.
 }

 \vspace{0.2cm}

The organization of this paper is as follows. In section 2, we give some preliminaries on direct integrals of separable Hilbert spaces and von Neumann algebras acting on separable Hilbert spaces. In section 3, we give a characterization of type II$_{1}$ von Neumann algberas with Property $\Gamma$ and obtain some equivalent definitions. In section 4, by showing that every finite subset $F$ of of a type II$_{1}$ von Neumann algebra $\mathcal{M}$ with Property $\Gamma$ is contained in a separable unital C$^*$-subalgebra with Property c$^*$-$\Gamma$, we obtain that $d(\mathcal{M}) =3$.

\section{Preliminaries}

\subsection{Dixmier Approximation Theorem}

We will need the following Dixmier Approximation Theorem in the
paper.

\begin{lemma}\label{Dixmier} (Dixmier Approximation Theorem) Let $\mathcal M$ be a finite von
Neumann algebra with center $\mathcal Z$. Let $\tau$ be the
center-valued trace on $\mathcal M$. If  $a\in\mathcal M$, then
$$
\{ \tau(a) \}=\mathcal Z \cap  \left (conv_{\mathcal M}(a)^{=} \right ),
$$  where
$conv_{\mathcal M}(a)^{=}$ is  the norm closure of the convex hull of
the set $\{uau^*: u \text{ is a unitary in } \mathcal M\}$.

\end{lemma}
\subsection{Direct integral theory}

  General knowledge about direct integrals of
separable Hilbert spaces and von Neumann algebras acting on
separable Hilbert spaces can be found in \cite{vN} and \cite{KR1}.
Here we list a few  lemmas that will be needed in this paper.
\begin{lemma} \label{2.1}
(\cite{KR1}) Suppose $\mathcal{M}$ is a von Neumann algebra acting on
a separable Hilbert space $H$. Let $\mathcal{Z}$ be the center of
$\mathcal{M}$. Then there is a direct integral decomposition of
$\mathcal{M}$ relative to $\mathcal{Z}$, i.e. there exists a locally
compact complete separable metric measure space $(X, \mu)$ such that
\begin{enumerate}
\item [(i)] $H$ is (unitarily equivalent to) the direct integral of $\{ H_{s} : s \in X \}$ over $(X, \mu)$, where each $H_{s}$ is a separable Hilbert space, $s \in X$.
\item [(ii)] $\mathcal{M}$ is (unitarily equivalent to) the direct integral
of $\{ \mathcal{M}_{s} \}$ over $(X, \mu)$, where $\mathcal{M}_{s}$
is a factor in $B(H_{s})$ almost everywhere. Also, if $\mathcal{M}$
is of type $I_{n}$($n$ could be infinite), II$_{1}$, II$_{\infty}$
or $III$, then the components $\mathcal{M}_{s}$ are, almost
everywhere, of type  I$_{n}$, II$_{1}$, II$_{\infty}$ or III,
respectively.
\end{enumerate}
Moreover, the center $\mathcal{Z}$ is (unitarily equivalent to) the algebra of diagonalizable operators relative to this decomposition.
\end{lemma}

The following lemma gives a decomposition of a normal sate on a direct integral of von Neumann algebras.
\begin{lemma} \label{2.2}
(\cite{KR1}) If $H$ is the direct integral of separable Hilbert
spaces $\{ H_{s} \}$ over $(X, \mu)$, $\mathcal{M}$ is a
decomposable von Neumann algebra on $H$ (i.e., every operator in $\mathcal{M}$ is decomposable relative to the direct integral decomposition, see Definition 14.1.6 in \cite{KR1}) and $\rho$ is a normal
state on $\mathcal{M}$. There is a positive normal linear functional
$\rho_{s}$ on $\mathcal{M}_{s}$ for every $s \in X$ such that
$\rho(a)=\int_{X} \rho_{s}(a(s))d\mu$ for each $a$ in
$\mathcal{M}$. If $\mathcal{M}$ contains the algebra $\mathcal{C}$
of diagonalizable operators and $\rho \vert_{E\mathcal{M}E}$ is
faithful or tracial, for some projection $E$ in $\mathcal{M}$, then
$\rho_{s} \vert_{E(s)\mathcal{M}_{s}E(s)}$ is, accordingly,
faithful or tracial almost everywhere.
\end{lemma}

\begin{remark} \label{2.3}
 From the proof of Lemma 14.1.19 in \cite{KR1}, we obtain that if
$\rho=\sum\limits_{n=1}^{\infty}\omega_{y_{n}}$ on $\mathcal{M}$,
where $\{ y_{n} \}$ is a sequence of vectors in $H$ such that
$\sum\limits_{n=1}^{\infty}\Vert y_{n} \Vert ^{2}=1$ and
$\omega_{y}$ is defined on $\mathcal{M}$ such that
$\omega_{y}(a)=\langle ay, y \rangle$ for any $a \in \mathcal{M}, y
\in H$, then $\rho_{s}$ can be chosen to be
$\sum\limits_{n=1}^{\infty} \omega_{y_{n}(s)}$ for each $s \in X$.
\end{remark}

\section{Some equivalent definitions of type II$_1$ von Neumann algebras with Property $\Gamma$}
In this section, we will give some equivalent definitions of
Property $\Gamma$ for type II$_1$ von Neumann algebras.

Let us recall the following definition of Property $\Gamma$ for
general type II$_{1}$ von Neumann algebras   in \cite{QS}. {\em
Suppose $\mathcal{M}$ is a type II$_{1}$ von Neumann algebra with a
predual $\mathcal M_{\sharp}$. Suppose   $\sigma (\mathcal M,
\mathcal M_\sharp)$ is the weak-$*$ topology on $\mathcal M$ induced
from $\mathcal M_\sharp$. We say that $\mathcal{M}$ has Property
$\Gamma$ if and only if $\forall \ a_{1}, a_{2}, \dots, a_{k} \in
\mathcal{M}$ and $\forall \ n\in \Bbb N$, there exist a partially
ordered set $\Lambda$ and a family of projections $$\{ p_{i
\lambda}: 1\le i\le n; \lambda \in \Lambda \}\subseteq \mathcal{M}$$
satisfying
\begin{enumerate}
\item [(i)] For each $\lambda \in \Lambda$,   $p_{1 \lambda}, p_{2 \lambda}, \dots,
p_{n \lambda}$ are mutually orthogonal equivalent projections in $\mathcal M$ with sum $I$.
\item [(ii)] For each $1\le i\le n$ and $1\le j\le k$,
$$
\lim_{ \lambda } (p_{i \lambda}a_{j}-a_{j}p_{i \lambda})^*(p_{i \lambda}a_{j}-a_{j}p_{i \lambda}) =0 \qquad \text {in $\sigma(\mathcal M, \mathcal M_\sharp)$ topology.}$$
\end{enumerate}}

The following two lemmas are well-known. We include their proofs here for the purpose of completeness.
\begin{lemma}\label{3.0.5}
Suppose that $\mathcal M$ is a type II$_1$ von Neumann algebra. Then
the following are true.
\begin{enumerate}
\item [(a)] For any nonzero element $x\in\mathcal M$, there exists a normal tracial state $\rho$ on $\mathcal M$ such that $\rho(x^*x)\ne 0$.
\item [(b)] There exists a non-zero central projection $q$ of $\mathcal M$, such that $q\mathcal
M$ is a countably decomposable type II$_1$ von Neumann algebra. \end{enumerate}
\end{lemma}
\begin{proof} Assume that $\mathcal M$ acts on a Hilbert space
$H$.

(a). Let $\mathcal Z$ be the center of $\mathcal M$ and
$\tau$ be the unique, normal, faithful, center-valued trace on $\mathcal M$
such that $\tau(a)=a$ for all $a\in \mathcal Z$ (see Theorem 8.2.8
in \cite{KR1}). Let $x\in \mathcal M$ be a non-zero element. Then, from the fact that $\tau$ is faithful,
we know that $\tau(x^*x)\ne 0$. Let $\hat \rho$ be a normal state on $\mathcal Z$ such that $\hat \rho(\tau(x^*x))\ne 0$.
Now the normal state $\hat \rho$ on $\mathcal Z$
can be extended to a normal tracial state $\rho$ on $\mathcal M$ by
defining $\rho(a)=\hat\rho (\tau(a))$ for all $a\in \mathcal M$.  Therefore $\rho$ is  a normal tracial state  on $\mathcal M$ such that $\rho(x^*x)\ne 0$.

(b). Let $\rho$ be a normal tracial state on $\mathcal{M}$ and
$\mathcal{I} = \{ a \in \mathcal{M}: \rho(a^*a) =0 \}$. Thus
$\mathcal{I}$ is a $2$-sided ideal in $\mathcal{M}$. Let $(H_{\rho}, \pi_{\rho}, x_{\rho})$ be the triple obtained from the GNS construction of
$\rho$ such that $H_{\rho}$ is a Hilbert space, $\pi_{\rho}: \mathcal{M} \to B(H_{\rho})$ is a $*$-representation, $x_{\rho} \in B(H_{\rho})$ is a cyclic vector for $\pi_{\rho}$ satisfying that,
for every $a \in \mathcal{M}$, $\rho(a) = \langle \pi_{\rho}(a)x_{\rho}, x_{\rho} \rangle$. It follows Proposition III.3.12
in \cite{Ta} that $\pi_{\rho}: \mathcal{M} \to B(H_{\rho})$ is a normal representation. Combining with the fact that $\mathcal{I} $ is a two-sided ideal, we can  check that $\mathcal{I} = \ker(\pi_{\rho})$.
Therefore $\mathcal I$ is   closed in $\mathcal M$ in ultraweak
operator topology. By Proposition 1.10.5 in \cite{Sa}, there exists
a central projection $q$ in $\mathcal Z$ such that $\mathcal
I=(1-q)\mathcal M$.

Now we claim that $q\mathcal M$ is countably decomposable. Suppose
that there is a collection of nonzero orthogonal projections $\{
q_{\alpha}: \alpha \in J \}$ in $q\mathcal M$ such that $q=
\sum_\alpha{q_{\alpha}}$.  Since $\rho$ is a normal tracial state on
$\mathcal M$, we know that $\rho(q)=\sum_{\alpha} \rho(q_\alpha)$. By the
definition of the ideal $\mathcal I$ and the choice of the central
projection $q$, we get that $\rho(q_{\alpha}) > 0$ for each $\alpha \in
J$. Now $1=\rho(q+(I-q))=\rho(q)=\sum_{\alpha  } \rho(q_{\alpha})$, where
$I$ is the identity of $\mathcal M$. It follows that $J$ is a
countable set and thus $q \mathcal{M}$ is countably decomposable.
\end{proof}

\begin{lemma}\label{3.0}
Suppose $\mathcal{M}$ is a type II$_{1}$ von Neumann algebra. Then
there is a family of orthogonal central projections $\{ q_{\alpha}:
\alpha \in \Omega\}$ in $\mathcal{M}$ with sum $I$ such that
$q_{\alpha}\mathcal{M}$ is countably decomposable for each $\alpha
\in \Omega$.
\end{lemma}
\begin{proof}

By Lemma \ref{3.0.5} and Zorn's lemma, there exists  an orthogonal
family $\{q_\alpha\}$ of non-zero central projections in $\mathcal
M$, which is maximal with respect to the property that
$q_\alpha\mathcal M$ is countable decomposable for each $\alpha$.
Let $Q=\sum q_{\alpha}$. We claim that $Q=I$, where $I$ is the
identity of $\mathcal M$. Assume, to the contrary, that $Q\ne I$.
Then by Lemma \ref{3.0.5}, there is a nonzero central projection $q$
in $(I-Q) \mathcal{M}$ such that $q \mathcal{M}$ is countably
decomposable. The existence of such $q$ contradicts with the
maximality of the family $\{q_\lambda\}$. Therefore $I=\sum
q_\alpha$ and the proof of the lemma is completed.
\end{proof}

\begin{remark} \label{rem 3.1}
Suppose $\mathcal{M}$ is a type II$_{1}$ von Neumann algebra with Property $\Gamma$.  Let
$  q$ be a central
projection of $\mathcal{M}$.  Then it follows
directly from the definition of Property $\Gamma$ that  $q
\mathcal{M}$ also has Property $\Gamma$.
\end{remark}

\begin{lemma} \label{lem 3.1}
Let $\mathcal{M}$ be a type II$_{1}$ von Neumann algbera acting on a
separable Hilbert space $H$ and $\mathcal{Z}_{\mathcal{M}}$ the center
of $\mathcal{M}$. Let $\tau$ be the center-valued trace on
$\mathcal{M}$ such that $\tau(z)=z$ for any $z \in
\mathcal{Z}_{\mathcal{M}}$. Let $\mathcal{M} = \int_{X} \bigoplus
\mathcal{M}_{s} d \mu$ and $H=\int_{X} \bigoplus H_{s} d \mu$ be the
direct integral decompositions of $\mathcal{M}$ and $H$ relative to
$\mathcal Z_{\mathcal M}$ as in Lemma \ref{2.1}. Assume that $\mathcal{M}_{s}$ is a type II$_{1}$
factor with a trace $\tau_{s}$ for each $s \in X$. Then for any $a
\in \mathcal{M}$,
$$\tau(a) (s) = \tau_{s}(a(s)) I_{s}$$
for almost every $s \in X$.
\end{lemma}
\begin{proof}
Fix $a \in \mathcal{M}$. By the Dixmier Approximation Theorem, for each $t \in \mathbb{N}$, there exist a positive integer $k_{t}$, a family of unitaries $\{ v_{j}^{(t)}: t \in \mathbb{N}, 1 \le j \le k_{t} \}$ in $\mathcal M$ and scalars $\{ \lambda_{j}^{(t)}: t \in \mathbb{N}, 1 \le j \le k_{t} \} \subseteq [0, 1]$ such that
\begin{enumerate}
\item[(i)] for each $t \in \mathbb{N}$, $\sum\limits_{1 \le j \le k_{t}} \lambda_{j}^{(t)} =1$;
\item[(ii)] $\lim\limits_{ t \to \infty} \Vert  \sum\limits_{1 \le j \le k_{t}} \lambda_{j}^{(t)} (v_{j}^{(t)})^* a v_{j}^{(t)} - \tau (a) \Vert=0$.
\end{enumerate}

Since $\{ v_{j}^{(t)}: t \in \mathbb{N}, 1 \le j \le k_{t} \}$ is a
countable set, we may assume that, for every $s \in X$,
$v_{j}^{(t)}(s)$ is  a unitary in $\mathcal{M}_{s}$ for any $t \in
\mathbb{N}$ and any $1 \le j \le k_{t}$. By Proposition 14.1.9 in
\cite{KR1}, for any $t \in \mathbb{N}$,  we have $$\Vert
\sum\limits_{1 \le j \le k_{t}} \lambda_{j}^{(t)} (v_{j}^{(t)})^* a
v_{j}^{(t)} - \tau (a) \Vert=\text{ess-}\sup_{s\in X} \Vert  \sum\limits_{1
\le j \le k_{t}} \lambda_{j}^{(t)} (v_{j}^{(t)}(s))^* a(s)
v_{j}^{(t)}(s) - \tau (a) (s) \Vert .$$ It follows that
\begin{eqnarray}
 \lim\limits_{ t \to \infty} \Vert  \sum\limits_{1 \le j \le k_{t}} \lambda_{j}^{(t)} (v_{j}^{(t)}(s))^* a(s) v_{j}^{(t)}(s) - \tau (a) (s) \Vert=0 \label{rem 3.1.1}
\end{eqnarray}
for almost every $s \in X$. Again, by the Dixmier Approximation
Theorem and the fact that each $\mathcal M_s$ is a type II$_1$ factor, (\ref{rem 3.1.1}) gives that $$\tau(a) (s) = \tau_{s}(a(s))
I_{s}$$ for almost every $s \in X$.
\end{proof}

\begin{lemma} \label{lem 3.2}
Let $\mathcal{M}$ be a type II$_{1}$ von Neumann algebra with center $\mathcal{Z}_{\mathcal{M}}$. Let $\tau$ be the center-valued trace on $\mathcal{M}$ such that $\tau(a)=a$ for any $a \in \mathcal{Z}_{\mathcal{M}}$. Suppose $\epsilon >0$, $x \in \mathcal{M}$ and $\tau(x^*x) < \epsilon I$. Then for any tracial state $\rho$ on $\mathcal{M}$,
$$\rho(x^*x) < 2\epsilon.$$
\end{lemma}
\begin{proof}
Note that $\tau(x^*x) < \epsilon I$. It follows from the Dixmier Approximation Theorem that there exist a positive integer $n \in \mathbb{N}$, a family of unitaries $\{ v_{1}, v_{2}, \dots, v_{n} \}$ in $\mathcal{M}$ and a family of scalars $\{ \alpha_{1}, \alpha_{2}, \dots, \alpha_{n} \} \subseteq [0, 1]$ such that
\begin{enumerate}
\item[(a)] $\sum\limits_{1 \le i \le n} \alpha_{i} =1$;
\item[(b)] $\Vert\tau(x^*x)- \sum\limits_{1 \le i \le n} \alpha_{i} v_{i}^* x^*x v_{i} \Vert < \epsilon$.
\end{enumerate}

Since $\rho$ is tracial, it follows from (a) and (b) that
$$\begin{aligned}\rho(x^*x) &=  \rho (\sum\limits_{1 \le i \le n} \alpha_{i} v_{i}^* x^*x v_{i})
\\ &= \rho
(\sum\limits_{1 \le i \le n} \alpha_{i} v_{i}^* x^*x
v_{i})-\tau(x^*x))+\rho(\tau(x^*x)) \\ &< 2\epsilon.\end{aligned}$$
The proof is completed.
\end{proof}

\begin{proposition} \label{3.2}
Suppose $\mathcal{M}$ is a type II$_{1}$ von Neumann algebra acting
on a separable Hilbert space $H$. Let $\tau$ be the center-valued
trace on $\mathcal{M}$ such that $\tau(a)=a$ for any $a \in
\mathcal{Z}_{\mathcal{M}}$, where $\mathcal{Z}_{\mathcal{M}}$ is the
center of $\mathcal{M}$.
Suppose that $\mathcal{M}$ has Property $\Gamma$.
Then, for $a_{1}, a_{2}, \dots, a_{k} \in \mathcal{M}$, any $n \in
\mathbb{N}$, any $\epsilon >0$, there exist $n$ orthogonal
equivalent projections $p_{1}, p_{2}, \dots, p_{n}$ in $\mathcal{M}$
with sum $I$ such that
$$\tau((p_{i}a_{j}-a_{j}p_{i})^*(p_{i}a_{j}-a_{j}p_{i})) < \epsilon I, \qquad \forall \quad 1 \le i \le n, 1 \le j \le k.$$
\end{proposition}
\begin{proof}
%Suppose $\mathcal{M}=\int_{X} \bigoplus \mathcal{M}_{s} d\mu$ and $H=\int_{X} \bigoplus H_{s} d\mu$ are the direct integral decompositions of $\mathcal{M}$ and $H$ over $(X, \mu)$ relative to
%$\mathcal{Z}_{\mathcal{M}}$. By Lemma \ref{2.1}, we may assume that
%$\mathcal{M}_{s}$ is a type II$_{1}$ factor with a trace $\tau_{s}$
%for each $s \in X$.

%(1) $\Leftrightarrow$ (2) $\Leftrightarrow$ (2') $\Leftrightarrow$ (3): It follows from Corollary 3.4 in \cite{QS} and Proposition 3.5 in \cite{QS2}.

 Suppose $\mathcal{M}$ has Property $\Gamma$.
Let $\mathcal{M} = \int_{X} \bigoplus \mathcal{M}_{s} d \mu$ and
$H=\int_{X} \bigoplus H_{s} d \mu$ be the direct integral
decompositions of $\mathcal{M}$ and $H$ relative to the center $\mathcal Z_{\mathcal M}$ as in Lemma \ref{2.1}.
We might assume that $\mathcal{M}_{s}$ is a type II$_{1}$ factor
with a trace $\tau_{s}$ for each $s \in X$.

 Fix $a_{1}, a_{2}, \dots, a_{k} \in \mathcal{M}$, $n \in \mathbb{N}$,
and $\epsilon >0$. By Corollary 4.2 in \cite{QS}, there exist $n$
orthogonal equivalent projections $p_{1}, p_{2}, \dots, p_{n}$ in
$\mathcal{M}$ with sum $I$ such that
\begin{eqnarray}
\Vert p_{i}(s)a_{j}(s)-a_{j}(s)p_{i}(s) \Vert_{2,s} \le \epsilon/2,
\qquad \forall \quad 1 \le i \le n, 1 \le j \le k, \label{3.2.1}
\end{eqnarray}
for almost every $s \in X$, where $\| \cdot \|_{2,s}$ is the trace
norm induced by $\tau_s$ on   $\mathcal M_s$ for each $s\in X$.

For any $1 \le i \le n, 1 \le j \le k$, Lemma \ref{lem 3.1} gives
\begin{eqnarray}
\tau((p_{i}a_{j}-a_{j}p_{i})^*(p_{i}a_{j}-a_{j}p_{i})) (s)= \tau_{s}((p_{i}(s)a_{j}(s)-a_{j}(s)p_{i}(s))^*(p_{i}(s)a_{j}(s)-a_{j}(s)p_{i}(s)))I_{s} \label{3.2.2}
\end{eqnarray}
for almost every $s \in X$.

For any $1 \le i \le n, 1 \le j \le k$, from (\ref{3.2.1}),
(\ref{3.2.2}) and Proposition 14.1.9 in \cite{KR1},  it follows
that $$\Vert \tau((p_{i}a_{j}-a_{j}p_{i})^*(p_{i}a_{j}-a_{j}p_{i}))
\Vert \le \epsilon/2$$ and, thus,
$$\tau((p_{i}a_{j}-a_{j}p_{i})^*(p_{i}a_{j}-a_{j}p_{i})) < \epsilon I.$$This finishes the proof.\end{proof}
%(4) $\Rightarrow$ (4'): This is clear.

%(4') $\Rightarrow$ (2'): This follows directly from Lemma \ref{lem 3.2}.

%(5) $\Rightarrow$ (3): It follows directly from Lemma \ref{lem 3.2}.

%(4) $\Rightarrow$ (5): Assume Condition (4) holds. Fix $\epsilon >0$ and $a_{1}, a_{2}, \dots, a_{k} \in \mathcal{M}$. By Condition (4), there exist two orthogonal equivalent projections $p_{1}, p_{2}$ with sum $I$ such that, for any $1 \le i \le 2$, $1 \le j \le k$,
%\begin{eqnarray}
%\tau((p_{i}a_{j}-a_{j}p_{i})^*(p_{i}a_{j}-a_{j}p_{i})) < \epsilon
%I/4. \label{3.2.3}
%\end{eqnarray}

%Now let $u=p_{1}-p_{2}$. It follows that $\tau(u)=\tau(p_{1})-\tau(p_{2})=0$. By (\ref{3.2.3}), for each $1 \le j \le k$,
%\begin{eqnarray*}
%&&\tau((ua_{j}-a_{j}u)^*(ua_{j}-a_{j}u)) \\
%&&=\tau(((p_{1}a_{j}-a_{j}p_{1})+(p_{2}a_{j}-a_{j}p_{2}))^*((p_{1}a_{j}-a_{j}p_{1})+(p_{2}a_{j}-a_{j}p_{2}))) \\
%&& \le 2 (\tau((p_{1}a_{j}-a_{j}p_{1})^*(p_{1}a_{j}-a_{j}p_{1}))+\tau((p_{2}a_{j}-a_{j}p_{2})^*(p_{2}a_{j}-a_{j}p_{2}))) \\
%&& < \epsilon I.
%\end{eqnarray*}

\begin{lemma}\label{mylemma 3.1}
Let $\mathcal M$ be a type II$_1$ von Neumann algebra with a center
$\mathcal Z_{\mathcal M}$. Let $\mathcal M_1$ be a von Neumann
subalgebra of $\mathcal M$ and $\mathcal Z_{\mathcal M_1}$ be the
center of $\mathcal M_1$. Suppose $\tau_{\mathcal M}$ and
$\tau_{\mathcal M_1} $ are the center-valued traces of $\mathcal M$,
and $\mathcal M_1$ respectively. For any $x\in \mathcal M_1$, we
have $\|\tau_{\mathcal M}(x)\|\le  \|\tau_{\mathcal M_1}(x)\|.$
\end{lemma}
\begin{proof}
Let $x$ be an element in $\mathcal M_1$. For any $\epsilon>0$, by
the Dixmier  Approximation Theorem, there exist  a positive integer
$k $, a family of unitaries $\{ v_{j}: 1 \le j \le k  \}$ in $\mathcal{M}_{1}$ and
scalars $\{ \lambda_{j} :  1 \le j \le k \} \subseteq [0, 1]$ such
that (i) $\sum\limits_{1 \le j \le k } \lambda_{j}  =1$ and (ii)
$\Vert \sum\limits_{1 \le j \le k } \lambda_{j}   v_{j} ^* xv_{j}  -
\tau_{\mathcal M_1} (x) \Vert\le \epsilon$.

Hence, $$
\begin{aligned}
\|\tau_{\mathcal M}(x)\| &= \|\tau_{\mathcal M} \left
(\sum\limits_{1 \le j \le k } \lambda_{j}   v_{j} ^* xv_{j} \right )
\| \\
&\le \|\tau_{\mathcal M} \left (\sum\limits_{1 \le j \le k }
\lambda_{j}   v_{j} ^* xv_{j} \right ) -\tau_{\mathcal
M}(\tau_{\mathcal M_1} (x)) \|  + \|\tau_{\mathcal
M}(\tau_{\mathcal M_1} (x)) \|\\
&\le \epsilon +  \|\tau_{\mathcal M_1}(x)\|
\end{aligned}
$$ Since $\epsilon $ is arbitrary, we have $\|\tau_{\mathcal M}(x)\|\le  \|\tau_{\mathcal M_1}(x)\|.$

\end{proof}

\begin{proposition} \label {mylemma 3.2}
Suppose $\mathcal{M}$ is a countably decomposable type II$_{1}$ von
Neumann algebra. Let $\tau$ be the center-valued trace on
$\mathcal{M}$ such that $\tau(a)=a$ for any $a \in
\mathcal{Z}_{\mathcal{M}}$, where $\mathcal{Z}_{\mathcal{M}}$ is the
center of $\mathcal{M}$.
%\begin{enumerate}
Suppose that $\mathcal{M}$ has Property $\Gamma$. Then, for $a_{1},
a_{2}, \dots, a_{k} \in \mathcal{M}$, any $n \in \mathbb{N}$, any
$\epsilon >0$, there exist $n$ orthogonal equivalent projections
$p_{1}, p_{2}, \dots, p_{n}$ in $\mathcal{M}$ with sum $I$ such that
$$\tau((p_{i}a_{j}-a_{j}p_{i})^*(p_{i}a_{j}-a_{j}p_{i})) < \epsilon I, \qquad \forall \quad 1 \le i \le n, 1 \le j \le k.$$
\end{proposition}
\begin{proof}
Let $a_{1}, a_{2}, \dots, a_{k}$ be in $\mathcal{M}$. By Lemma 3.6
in \cite{QS2},  there is a type II$_{1}$ von Neumann algebra
$\mathcal{M}_{1}$ with separable predual and Property $\Gamma$ such
that $\{a_1,\ldots, a_k\} \subseteq \mathcal{M}_{1} \subseteq
\mathcal{M}$. From Proposition \ref{3.2}, it follows that  there exist $n$ orthogonal equivalent projections
$p_{1}, p_{2}, \dots, p_{n}$ in $\mathcal{M}_1$ with sum $I$ such that
$$\tau_{\mathcal M_1}((p_{i}a_{j}-a_{j}p_{i})^*(p_{i}a_{j}-a_{j}p_{i})) < \epsilon I, \qquad \forall \quad 1 \le i \le n, 1 \le j \le k,$$ where
$\tau_{\mathcal M_1}$ is the center-valued trace on $\mathcal M_1$. By
Lemma \ref{mylemma 3.1}, we obtain that
$$\tau((p_{i}a_{j}-a_{j}p_{i})^*(p_{i}a_{j}-a_{j}p_{i})) < \epsilon I, \qquad \forall \quad 1 \le i \le n, 1 \le j \le k.$$
\end{proof}

\begin{remark} \label{rem 3.3}
Suppose $\mathcal{M}$ is a type II$_{1}$ von Neumann algbera with center $\mathcal{Z}_{\mathcal{M}}$. Let $\tau$ be the center-valued trace on $\mathcal{M}$ such that $\tau(a) = a$ for any $a \in \mathcal{Z}_{\mathcal{M}}$. Suppose $\{ q_{\alpha}: \alpha \in \Omega \}$ is a family of nonzero orthogonal central projections in $\mathcal{M}$ with sum $I$. Therefore $q_{\alpha} \mathcal{M}$ is a type II$_{1}$ von Neumann algebra with center $q_{\alpha} \mathcal{Z}_{\mathcal{M}}$. Let $\tau_{\alpha}$ be the center-valued trace on $q_{\alpha}\mathcal{M}$ such that $\tau_{\alpha}(a)=a$ for any $a \in q_{\alpha} \mathcal{Z}_{\mathcal{M}}$. We have
$$\tau(a) = \sum\limits_{\alpha \in \Omega} \tau_{\alpha} (q_{\alpha} a), \qquad \forall \quad a \in \mathcal{M}. $$
\end{remark}

\begin{theorem} \label{3.3}
Suppose $\mathcal{M}$ is a type II$_{1}$ von Neumann algebra and
 $\mathcal{Z}_{\mathcal{M}}$ is the center of $\mathcal{M}$.
Let $\tau$ be the center-valued  trace on $\mathcal{M}$ such that
$\tau(a)=a$ for any $a \in \mathcal{Z}_{\mathcal{M}}$. Then the
following statements are equivalent:
\begin{enumerate}       \item[(i)] $\mathcal{M}$ has Property $\Gamma$.
\item[(ii)] There exists a family of nonzero orthogonal central projections $\{ q_{\alpha}: \alpha \in \Omega \}$ in $\mathcal M$ with sum $I$ such that $q_{\alpha} \mathcal{M}$ is a countably decomposable type II$_{1}$ von Neumann algbera with Property $\Gamma$ for each $\alpha \in
\Omega$.
\item[(iii)] For any $n \in \mathbb{N}$, any $\epsilon >0$ and $a_{1}, a_{2}, \dots, a_{k} \in \mathcal{M}$, there exist $n$ orthogonal equivalent projections $p_{1}, p_{2}, \dots, p_{n}$ in $\mathcal{M}$ with sum $I$ such that
$$\tau((p_{i}a_{j}-a_{j}p_{i})^*(p_{i}a_{j}-a_{j}p_{i})) < \epsilon I, \qquad \forall \quad 1 \le i \le n, 1 \le j \le k.$$
\item[(iv)] There exists a positive integer $n_{0} \ge 2$ satisfying {\em for any $\epsilon >0$ and $a_{1}, a_{2}, \dots, a_{k} \in \mathcal{M}$, there exist $n_{0}$ orthogonal equivalent projections $p_{1}, p_{2}, \dots, p_{n_{0}}$ in $\mathcal{M}$ with sum $I$ satisfying
$$\tau((p_{i}a_{j}-a_{j}p_{i})^*(p_{i}a_{j}-a_{j}p_{i})) < \epsilon I, \qquad \forall \quad 1 \le i \le n_{0}, 1 \le j \le
k.$$}
\item[(v)] For any $\epsilon >0$ and $a_{1}, a_{2}, \dots, a_{k} \in \mathcal{M}$, there exists a unitary $u$ in $\mathcal{M}$ such that
\begin{enumerate}
\item[(a)] $\tau(u)=0$;
\item[(b)] $\tau((ua_{j}-a_{j}u)^*(ua_{j}-a_{j}u)) < \epsilon I, \qquad \forall \quad 1 \le j \le k.$
\end{enumerate}
\item[(vi)] For  any $n \in \mathbb{N}$, any normal tracial state $\rho$ on $\mathcal{M}$, any $\epsilon > 0$ and $a_{1}, a_{2}, \dots, a_{k} \in \mathcal{M}$, there exist $n$ orthogonal equivalent projections $p_{1}, p_{2}, \dots, p_{n}$ in $\mathcal{M}$ with sum $I$ such that
$$\Vert p_{i}a_{j}-a_{j}p_{i} \Vert_{2, \rho} < \epsilon, \qquad \forall \quad 1\le i \le n, 1 \le j \le k,$$
where $\Vert \cdot \Vert_{2, \rho}$ is the $2$-norm on $\mathcal{M}$ induced by $\rho$.
\item[(vii)] There exists a positive integer $n_{0} \ge 2$ satisfying {\em for any  normal tracial state $\rho$ on $\mathcal{M}$, any $\epsilon > 0$ and $a_{1}, a_{2}, \dots, a_{k} \in \mathcal{M}$, there exist $n_{0}$ orthogonal equivalent projections $p_{1}, p_{2}, \dots, p_{n_{0}}$ in $\mathcal{M}$ with sum $I$ satisfying
$$\Vert p_{i}a_{j}-a_{j}p_{i} \Vert_{2, \rho} < \epsilon, \qquad \forall \quad 1\le i \le n_{0}, 1 \le j \le
k,$$}
where $\Vert \cdot \Vert_{2, \rho}$ is the $2$-norm on $\mathcal{M}$ induced by $\rho$.
\item[(viii)] For any normal tracial state $\rho$ on $\mathcal{M}$, any $\epsilon >0$ and $a_{1}, a_{2}, \dots, a_{k} \in \mathcal{M}$, there exists a unitary $u$ in $\mathcal{M}$ such that
\begin{enumerate}
\item[(a)]  $\tau(u)=0$;
\item[(b)]  $\Vert ua_{j}-a_{j}u \Vert_{2, \rho} < \epsilon,$  for all $ 1 \le j \le k,$
where $\Vert \cdot \Vert_{2, \rho}$ is the $2$-norm on $\mathcal{M}$ induced by $\rho$.
\end{enumerate}
\end{enumerate}
\end{theorem}
\begin{proof} We will prove the result by showing that
(i)$\Rightarrow$ (ii) $\Rightarrow$ (iii) $\Rightarrow$ (iv) $\Rightarrow$
(v)$\Rightarrow$ (ii), (iii)$\Rightarrow$ (i) and (iii) $\Rightarrow$
(vi) $\Rightarrow$ (vii) $\Rightarrow$ (viii) $\Rightarrow$ (ii).

(i) $\Rightarrow$ (ii): It follows from Lemma \ref{3.0} and Remark
\ref{rem 3.1}.

(ii) $\Rightarrow$ (iii): Assume that there exists a family of
nonzero orthogonal central projections $\{ q_{\alpha}: \alpha \in
\Omega \}$ with sum $I$ such that $q_{\alpha} \mathcal{M}$ is a
countably decomposable type II$_{1}$ von Neumann algbera with
Property $\Gamma$ for each $\alpha \in \Omega$. Fix $n \in
\mathbb{N}$, any $\epsilon >0$ and $a_{1}, a_{2}, \dots, a_{k} \in
\mathcal{M}$. Then
$$
a_j= \sum_\alpha q_\alpha a_j, \qquad \forall \quad 1\le j\le n.
$$
For each $\alpha\in \Omega$, by Proposition \ref{mylemma 3.2}, there
exist $n$ orthogonal equivalent projections $p_{1}^{(\alpha)},
p_{2}^{(\alpha)}, \dots, p_{n}^{(\alpha)}$ in $q_\alpha\mathcal{M}$
with sum $q_\alpha$ such that
\begin{equation}\tau_{\alpha}((p_{i}^{(\alpha)}(q_\alpha a_{j})-(q_\alpha a_{j})p_{i}^{(\alpha)})^*
(p_{i}^{(\alpha)}(q_\alpha a_{j})-(q_\alpha a_{j})p_{i}^{(\alpha)}))
< \epsilon \cdot q_\alpha, \label{myequ 3.1}\end{equation}  for all
$ 1 \le i \le n, 1 \le j \le k$, where $\tau_\alpha$ is the
center-valued trace on $q_\alpha\mathcal M$. Let
$$p_i=\sum_{\alpha} p_{i}^{(\alpha)}, \qquad \text { for all } 1\le
i\le n.$$
Then it is not hard to see that $p_1,\ldots, p_n$ are orthogonal equivalent projections in $\mathcal M$ with sum
$I$. By Remark \ref{rem 3.3} and inequality (\ref{myequ 3.1}), we
know
$$\tau((p_{i}a_{j}-a_{j}p_{i})^*(p_{i}a_{j}-a_{j}p_{i})) < \epsilon I, \qquad \forall \quad 1 \le i \le n, 1 \le j \le k.$$

(iii) $\Rightarrow$ (iv): It is obvious.

(iv) $\Rightarrow$ (v): Assume that there exists a positive integer
$n_{0} \ge 2$ satisfying {\em for any $\epsilon >0$ and $a_{1},
a_{2}, \dots, a_{k} \in \mathcal{M}$, there exist $n_{0}$ orthogonal
equivalent projections $p_{1}, p_{2}, \dots, p_{n_{0}}$ in
$\mathcal{M}$ with sum $I$ satisfying
$$\tau((p_{i}a_{j}-a_{j}p_{i})^*(p_{i}a_{j}-a_{j}p_{i})) < \frac {\epsilon} {n_{0}^{2}} I, \qquad \forall \quad 1 \le i \le n_{0}, 1 \le j \le
k.$$} Let $\lambda =e^{2\pi i/n_0}$ be the $n_0$-th root of unit.
Let
$$
u=p_1+\lambda p_2+\cdots \lambda^{n_0-1}p_{n_{0}}.
$$ Since $p_1, \ldots, p_{n_0}$ are orthogonal equivalent
projections in $\mathcal M$, we know $\tau(u)=0$. A quick
computation shows that
$$\tau((ua_{j}-a_{j}u)^*(ua_{j}-a_{j}u)) <   \epsilon   I.$$

(v)$\Rightarrow$ (ii): Assume that (v) holds. From Lemma \ref{3.0},
there is a family of orthogonal central projections $\{ q_{\alpha}:
\alpha \in \Omega\}$ in $\mathcal{M}$ with sum $I$ such that
$q_{\alpha}\mathcal{M}$ is countably decomposable for each $\alpha
\in \Omega$.

Next we will show that $q_\alpha \mathcal M$ has Property $\Gamma$ for each $\alpha$ in $\Omega$. Let $x_1,\ldots, x_k$ be elements in $q_\alpha\mathcal M$. By the assumption (v), for any $\epsilon >0$, there exists a unitary $u$ in $\mathcal{M}$ such that (a)  $\tau(u)=0$ and (b)
 $\tau((ua_{j}-a_{j}u)^*(ua_{j}-a_{j}u)) < \epsilon I,$ for all $ 1 \le j \le k$. Since $q_\alpha \mathcal M$ is a countably decomposable type II$_1$ von Neumann subalgebra, there exists a faithful normal tracial state $\rho$ on $q_\alpha \mathcal M$. We can naturally extend $\rho$ on $q_\alpha \mathcal M$ to a normal tracial state $\tilde \rho$ on $\mathcal M$ by defining $\tilde\rho(x)=\rho(q_\alpha x)$ for all $x$ in $\mathcal M$.
 It is not hard to see that $q_\alpha u$ is a unitary in $q_\alpha \mathcal M$ and $\tau_{\alpha}(q_\alpha u)=\tau(q_\alpha u)= q_\alpha\tau(u)=0$, where $\tau_\alpha$ is a center-valued trace on $q_\alpha\mathcal M$. Moreover, by the Dixmier Approximation Theorem, we have
 $$
 \tilde \rho (x) = \tilde \rho(\tau(x)), \qquad \forall \ x\in \mathcal M.
 $$
Hence $$\begin{aligned}
 \rho(((q_\alpha u)a_{j}-a_{j}(q_\alpha u))^*((q_\alpha u)a_{j}-a_{j}(q_\alpha u))) &=
  \tilde\rho\left ((ua_{j}-a_{j}u)^*(ua_{j}-a_{j}u)\right )\\& =\tilde\rho\left (\tau((ua_{j}-a_{j}u)^*(ua_{j}-a_{j}u))\right )\\
  &\le \epsilon,
 \end{aligned}$$  for all $1\le i\le k.$ By Proposition 3.5 in \cite{QS2}, we conclude that $q_\alpha\mathcal M$ has Property $\Gamma$.

(iii)$\Rightarrow$ (i): Assume that (iii) is true. We assume that $\mathcal M$ acts on a Hilbert space $H$. Let $x_1,\ldots, x_k$ be a family of elements in $\mathcal M$. From (iii), for any positive integer $n$, there exists a family of  projections $\{p_{ir}: 1\le i\le n, r\ge 1\}$ in $\mathcal{M}$ such that
\begin{enumerate}
\item [1.] For each $r\ge 1$, $p_{1,r}, \ldots, p_{n,r}$ are orthogonal equivalent projections in $\mathcal M$ with sum $I$.
\item [2.] Moreover,
$$\lim_{r\rightarrow \infty}\|\tau((p_{i,r}a_{j}-a_{j}p_{i,r})^*(p_{i,r}a_{j}-a_{j}p_{i,r}))\| =0, \qquad \forall \  1\le i\le n,  1 \le j \le k.$$
\end{enumerate}
Thus, for any normal tracial state $\rho$ on $\mathcal M$, we have
\begin{align}
  \lim_{r\rightarrow \infty}\| \rho((p_{i,r}a_{j} -a_{j}p_{i,r})^* (p_{i,r}a_{j}-a_{j}p_{i,r})) \|
  &= \lim_{r\rightarrow \infty}\|\rho(\tau((p_{i,r}a_{j}-a_{j}p_{i,r})^*(p_{i,r}a_{j}-a_{j}p_{i,r})))\| \notag\\
  &\le \lim_{r\rightarrow \infty}\| \tau((p_{i,r}a_{j}-a_{j}p_{i,r})^*(p_{i,r}a_{j}-a_{j}p_{i,r}))\notag\|\\
  &=0.  \label{myeq 3.2}
\end{align}
Let $\{\rho_{\lambda}\}_{\lambda\in\Lambda}$ be the collection of all normal tracial states on $\mathcal M$. For each $\lambda\in\Lambda$, let $(\pi_\lambda,   H_\lambda, \hat I_\lambda)$ be the GNS representation, obtained from $\rho_\lambda$,  of $\mathcal M$ on the Hilbert space $H_\lambda=L^2(\mathcal M,\rho_\lambda)$ with a cyclic vector $\hat I_\lambda$ in $H_\lambda$. We also let $K=\sum_{\lambda\in \Lambda}   H_\lambda$ be the direct sum of Hilbert spaces $\{H_{\lambda}\}$ and $\pi =\sum_{\lambda\in \Lambda} \pi_{\lambda}:\mathcal M\rightarrow B(K)$ be the direct sum of $\{\pi_\lambda\}$. Thus $\pi$ is a  $*$-representation of $\mathcal M$ on $K$ defined by
$$
\pi(x)((\xi_\lambda))=(\pi_{\lambda}(x)\xi_\lambda ), \qquad \forall \ (\xi_\lambda)\in \sum_{\lambda\in \Lambda}  H_\lambda= K.
$$
It is not hard to see that $\pi$ is a normal $*$-representation and $\pi(\mathcal M)$ is also a von Neumann algebra.
By Lemma \ref{3.0.5} (a),  $\pi$ is a $*$-isomorphism from $\mathcal M$ onto $\pi(\mathcal M)$.

We claim  that,  for all $ 1\le i\le n,  1 \le j \le k,$
$$
 (p_{i,r}a_{j} -a_{j}p_{i,r})^* (p_{i,r}a_{j}-a_{j}p_{i,r}) \rightarrow 0 \  \text {in ultraweak operator topology (or in  $\sigma(\mathcal M, \mathcal M_\sharp)$ topology).}
$$ Actually, the claim is equivalent  to the statement that
$$
 \pi((p_{i,r}a_{j} -a_{j}p_{i,r})^* (p_{i,r}a_{j}-a_{j}p_{i,r})) \rightarrow 0 \  \text {in ultraweak topology}.
$$ Note that $$\{(p_{i,r}a_{j} -a_{j}p_{i,r})^* (p_{i,r}a_{j}-a_{j}p_{i,r}):  1\le i\le n,  1 \le j \le k, r \in\Bbb N\}$$ is a bounded subset in $\mathcal M$. It will be enough if we are able to show that
$$
 \pi((p_{i,r}a_{j} -a_{j}p_{i,r})^* (p_{i,r}a_{j}-a_{j}p_{i,r})) \rightarrow 0 \ \ \  \text {in  weak operator topology},
$$ or
\begin{equation}
 \pi (p_{i,r}a_{j}-a_{j}p_{i,r} ) \rightarrow 0 \ \ \  \text {in  strong operator topology}. \label{myequ 3.3}
\end{equation}
By the construction of $\pi$, (\ref{myequ 3.3}) follows directly from (\ref{myeq 3.2}).

 From the claim in the preceding paragraph, by the definition of Property $\Gamma$, we know that $\mathcal M$ has property $\Gamma$.

(iii)$\Rightarrow$ (vi): From the Dixmier Approximation Theorem, for any normal tracial state $\rho$ on $\mathcal M$, we have
$$
\rho(x)=\rho(\tau(x))  \qquad \forall \ x\in \mathcal M.
$$
Now (vi) follows easily from (iii).

(vi) $\Rightarrow$ (vii): It is obvious.

(vii) $\Rightarrow$ (viii): It is similar to (iv)$\Rightarrow$ (v).

(viii) $\Rightarrow$ (ii): Assume that (viii) holds.  From Lemma \ref{3.0},
there is a family of orthogonal central projections $\{ q_{\alpha}:
\alpha \in \Omega\}$ in $\mathcal{M}$ with sum $I$ such that
$q_{\alpha}\mathcal{M}$ is countably decomposable for each $\alpha
\in \Omega$. We need to show that $q_\alpha \mathcal M$ has Property $\Gamma$ for each $\alpha$ in $\Omega$.

Since each $q_\alpha \mathcal M$ is a countably decomposable type II$_1$ von Neumann algebra. There exists a faithful normal tracial state $\rho_\alpha$ on  $q_\alpha \mathcal M$. Then the normal tracial state $\rho_\alpha$ on  $q_\alpha \mathcal M$ can be naturally extended to a normal tracial state $\tilde \rho$ on $\mathcal M$ by defining $\tilde \rho(x)=\rho_\alpha(q_\alpha x)$ for all $x\in \mathcal M$.
Let $\epsilon>0$ and $a_1,\ldots, a_k$ be elements in $q_{\alpha} \mathcal M$. Since (viii) holds, there exists a unitary $u$ in $\mathcal{M}$ such that
\begin{enumerate}
\item[(a)] $\tau(u)=0$;
\item[(b)] $\Vert ua_{j}-a_{j}u \Vert_{2, \tilde \rho} < \epsilon, $ for all $ 1 \le j \le k,$ where $\|\cdot\|_{2, \tilde \rho}$ is the trace norm induced by $\tilde \rho$ on $\mathcal M$.
\end{enumerate}
Now it is not hard to verify that  $q_\alpha u$ is a unitary in $q_\alpha \mathcal M$ satisfying $\tau_{\alpha}(q_\alpha u)=\tau(q_\alpha u)=0$, where $\tau_{\alpha}$ is the unique center-valued trace on  $q_\alpha \mathcal M$. Moreover
$$
\begin{aligned}
 \Vert (q_\alpha u)a_{j}-a_{j}(q_\alpha u)\Vert_{2,  \rho_\alpha} & =  \Vert (q_\alpha u)a_{j}-a_{j}(q_\alpha u)\Vert_{2, \tilde \rho}\\
 &= \Vert ua_{j}-a_{j}u \Vert_{2, \tilde \rho} \\
 &<\epsilon.
\end{aligned}
$$
From Proposition 3.5 in \cite{QS2}, it follows that  $q_\alpha \mathcal M$ has Property $\Gamma$ for each $\alpha$ in $\Omega$. This ends the whole proof.
\end{proof}

\section{Similarity degree of type II$_{1}$ von Neumann algebras with Property $\Gamma$}
Let us recall a definition of Property c$^*$-$\Gamma$ for unital C$^*$-algebras given in \cite{QS2}. {\em Suppose $\mathcal{A}$ is a unital C$^*$-algebra. We say $\mathcal{A}$ has Property c$^*$-$\Gamma$ if it satisfies the following condition:
\begin{enumerate}\item[]
 If $\pi$ is a unital $*$-representation of $\mathcal{A}$ on a Hilbert space $H$ such that $\pi(\mathcal{A})''$ is a type II$_{1}$ factor, then $\pi (\mathcal{A})''$ has Property $\Gamma$.
\end{enumerate} }

If $\mathcal{A}$ is a separable unital C$^*$-algebra with Property c$^*$-$\Gamma$, Theorem 5.3 in \cite{QS2} gives that the similarity degree of $\mathcal{A}$ is no more than 3. Indeed, it was shown in Theorem 5.3 in \cite{QS2} that, for any C$^*$-algebra $\mathcal{B}$, if $\phi$ is a bounded unital homomorphism from $\mathcal{A}$ to $\mathcal{B}$, then $\Vert \phi \Vert_{cb} \le \Vert \phi \Vert^3$.

\begin{lemma} \label{4.1}
Suppose $\mathcal{M}$ is a type II$_{1}$ von Neumann algebra with Property $\Gamma$. Let $\tau$ be the center-valued trace on $\mathcal{M}$ such that $\tau(a)=a$ for any $a \in \mathcal{Z}_{\mathcal{M}}$, where $\mathcal{Z}_{\mathcal{M}}$ is the center of $\mathcal{M}$. Suppose $F$ is a finite subset of $\mathcal{M}$. Then there exists a separable unital C$^*$-subalgebra $\mathcal{A}$ with Property c$^*$-$\Gamma$ satisfying $F \subseteq \mathcal{A} \subseteq \mathcal{M}$.
\end{lemma}
\begin{proof}
Let $F_{1}=F=\{ x_{1}, x_{2}, \dots, x_{k} \}$ be a finite subset of $\mathcal{M}$.  Since $\mathcal{M}$ has Property $\Gamma$, by Theorem \ref{3.3},  there exists  a $2 \times 2$ system of matrix units $\{ e_{11}^{(1)}, e_{12}^{(1)}, e_{21}^{(1)}, e_{22}^{(1)}\}$ such that
\begin{enumerate}
  \item [(i$_1$)] $e_{11}^{(1)}+e_{22}^{(1)}=I$, where $I$ is the identity of $\mathcal M$.
  \item [(ii$_1$)] $\tau((e_{ii}^{(1)}x-xe_{ii}^{(1)})^*(e_{ii}^{(1)}x-xe_{ii}^{(1)}))\le  \frac 1 2    I$, for each $x\in F_{1}$.
\end{enumerate}
 From (ii$_1$), by the Dixmier  Approximation Theorem,
there exist a positive integer $n_{1}$, a family of unitaries $v_{1}^{(1)}, v_{2}^{(1)}, \dots, v_{n_{1}}^{(1)}$ in $\mathcal{M}$  such that
\begin{enumerate}
\item[(iii$_1$)] For each $1 \le i \le 2$ and each $x \in F_1$, there is an element $y$ in the convex hull of $\{ (v_{t}^{(1)})^* (e_{ii}^{(1)}x-xe_{ii}^{(1)})^*(e_{ii}^{(1)}x-xe_{ii}^{(1)}) v_{t}^{(1)} : 1 \le t \le n_{1} \}$ with $\Vert y \Vert < 1$.
\end{enumerate}
Let $ F_{2} = F_{1} \cup \{ e_{11}^{(1)}, e_{12}^{(1)}, e_{21}^{(1)}, e_{22}^{(1)}\} \cup \{ v_{1}^{(1)}, \dots, v_{n_{1}}^{(1)} \}. $

Assume that $F_{1} \subseteq F_{2} \subseteq \dots \subseteq F_{m}$ have been constructed for some $m \ge 2$.
Since $\mathcal{M}$ has Property $\Gamma$, again by Theorem \ref{3.3},  there exists  a $2 \times 2$ system of matrix units $\{ e_{11}^{(m)}, e_{12}^{(m)}, e_{21}^{(m)}, e_{22}^{(m)}\}$ such that
\begin{enumerate}
  \item [(i$_m$)] $e_{11}^{(m)}+e_{22}^{(m)}=I$, where $I$ is the identity of $\mathcal M$.
  \item [(ii$_m$)] $\tau((e_{ii}^{(m)}x-xe_{ii}^{(m)})^*(e_{ii}^{(m)}x-xe_{ii}^{(m)}))\le  \frac 1 {m+1}   I$, for each $x\in F_{m}$.
\end{enumerate}
 From (ii$_m$), by the Dixmier Approximation Theorem,
there exist a positive integer $n_{m}$, a family of unitaries $v_{1}^{(m)}, v_{2}^{(m)}, \dots, v_{n_{m}}^{(m)}$ in $\mathcal{M}$  such that
\begin{enumerate}
\item[(iii$_m$)] For each $1 \le i \le 2$ and each $x \in F_m$, there is an element $y$ in the convex hull of $\{ (v_{t}^{(m)})^* (e_{ii}^{(m)}x-xe_{ii}^{(m)})^*(e_{ii}^{(m)}x-xe_{ii}^{(m)}) v_{t}^{(m)} : 1 \le t \le n_{m} \}$ with $\Vert y \Vert < \frac 1 m$.
\end{enumerate}
Let $ F_{m+1} = F_{m} \cup \{ e_{11}^{(m)}, e_{12}^{(m)}, e_{21}^{(m)}, e_{22}^{(m)} \} \cup \{ v_{1}^{(m)}, v_{2}^{(m)}, \dots, v_{n_{m}}^{(m)} \}. $

Continuing this process, we are able to obtain a sequence $\{F_m\}$, a sequence of system of units $ \{ e_{11}^{(m)}, e_{12}^{(m)}, e_{21}^{(m)}, e_{22}^{(m)} \}$ such that
\begin{enumerate}
 \item [(0)] $\{x_1,\ldots, x_k\}=F_1\subseteq F_2\subseteq \cdots \subseteq F_m\subseteq \cdots$.
 \item [(1)] For each $m\ge 1$, $\{ e_{11}^{(m)}, e_{12}^{(m)}, e_{21}^{(m)}, e_{22}^{(m)} \} $ is a system of units such that $e_{11}^{(m)}+e_{22}^{(m)}=I$.
 \item [(2)]  $\tau((e_{ii}^{(m)}x-xe_{ii}^{(m)})^*(e_{ii}^{(m)}x-xe_{ii}^{(m)}))\le  \frac 1 {m+1}   I$, for each $x\in F_{m}$.
 \item [(3)] For each   $  i =1, 2$ and each $x \in F_m$, there is an element  $$\text{ $y \in conv\{ (v^*  (e_{ii}^{(m)}x-xe_{ii}^{(m)})^*(e_{ii}^{(m)}x-xe_{ii}^{(m)}) v: $     $ v$ is a unitary in $F_{m+1} \}$}$$ satisfying $$\Vert y \Vert < \frac 1 m.$$
\end{enumerate}

Let $\mathcal{A}$ be the unital C$^*$-algebra generated by $\cup_{m \in \mathbb{N}} F_{m}$. Then $\mathcal{A}$ is separable and it follows from the preceding construction that,
\begin{enumerate}
\item[(4)]  for $i=1,2$ and  any $x \in \mathcal{A}$, there exists a sequence of  elements $\{y_{m}\}_{m\ge 1}$ in $\mathcal A$ such that, for $m\ge 1$, each $y_m$  is in the convex hull of $\{ v^* (e_{ii}^{(m)}x-xe_{ii}^{(m)})^*(e_{ii}^{(m)}x-xe_{ii}^{(m)}) v : v \text{ is a unitary in } \mathcal{A}\}$ and $\lim\limits_{m \to \infty} \Vert y_{m} \Vert = 0$.
\end{enumerate}

Now we are going  to show this  C$^*$-subalgebra $\mathcal{A}$ of $\mathcal{M}$ has Property c$^*$-$\Gamma$.  Suppose $\pi $ is a unital $*$-representation of $\mathcal A$ on a Hilbert space $H$ such that $\pi(\mathcal{A})''$ is a type II$_{1}$ factor. Notice that for each $m \in \mathbb{N}$, $\{ e_{11}^{(m)}, e_{12}^{(m)}, e_{21}^{(m)}, e_{22}^{(m)} \}$ is a $2 \times 2$ system of matrix units in $\mathcal{A}$. It follows that $\{ \pi(e_{11}^{(m)}), \pi(e_{12}^{(m)}), \pi(e_{21}^{(m)}), \pi(e_{22}^{(m)}) \}$ is also a system of matrix units. Hence $\pi(e_{11}^{(m)}), \pi(e_{22}^{(m)})$ are orthogonal equivalent projections in $\pi(\mathcal{A})''$ with sum $I$.

It follows from Condition (4) that
\begin{enumerate}
\item[(4')] for $i=1,2$ and  any $x \in \pi(\mathcal{A})$, there exists a sequence of  elements $\{y_{m}\}_{m\ge 1}$ in $\pi(\mathcal A)$ such that, for $m\ge 1$, each $y_m$  is in the convex hull of $$\{ v^* (\pi(e_{ii}^{(m)})x-x\pi(e_{ii}^{(m)}))^*(\pi(e_{ii}^{(m)})x-x\pi(e_{ii}^{(m)})) v : \ \ v \text{ is a unitary in } \pi(\mathcal{A})\}$$ and $\lim\limits_{m \to \infty} \Vert y_{m} \Vert = 0$.
\end{enumerate}
Let $\rho$ be the unique trace on $\pi(\mathcal{A})''$. Since $\rho$ is tracial,   Condition (4') implies that, for any $x \in \pi(\mathcal{A})$,
\begin{eqnarray}
&&\lim\limits_{m \to \infty} \rho((\pi(e_{ii}^{(m)})\pi(x)-\pi(x) \pi (e_{ii}^{(m)})^*(\pi(e_{ii}^{(m)}) \pi(x)-\pi(x) \pi(e_{ii}^{(m)})) \notag \\
&&=\lim\limits_{m \to \infty} \rho(y_{m}) \notag \\
&&=0. \label{4.1.1}
\end{eqnarray}
By Kaplansky Density Theorem, it follows from (\ref{4.1.1}) that
$$\lim\limits_{m \to \infty}\rho((\pi(e_{ii}^{(m)}) a- a \pi (e_{ii}^{(m)})^*(\pi(e_{ii}^{(m)}) a - a \pi(e_{ii}^{(m)})) =0$$
for any $a \in \pi(\mathcal{A})''$.
Note that a type II$_1$ factor is always countably decomposable. By Proposition 3.5 in \cite{QS2}, $\pi(\mathcal{A})''$ has Property $\Gamma$, whence we conclude that $\mathcal{A}$ has Property c$^*$-$\Gamma$.

The proof is completed.
\end{proof}

It was shown in \cite{C2} that the similarity degree of a type II$_{1}$ factor with Property $\Gamma$ is 3. The following theorem gives a generalization.

\begin{theorem}\label{4.2}
If $\mathcal{M}$ is a type II$_{1}$ von Neumann algebra with Property $\Gamma$, then the similarity degree $d(\mathcal{M}) =3$.
\end{theorem}
\begin{proof}
Since $\mathcal{M}$ is a von Neumann algebra of type II$_{1}$, by Corollary 1.9 in \cite{SW}, it is not nuclear. It follows from Theorem 1 in \cite{Pi3} that $d(\mathcal{M}) \ge 3$. In the following we show that $d(\mathcal{M})$ is no more than 3.

Suppose $\phi: \mathcal{M} \to B(H)$ is a bounde unital homomorphism, where $H$ is a Hilbert space. We will show that $\Vert \phi \Vert_{cb} \le \Vert \phi \Vert^{3} $. In fact we are going to prove that, for any $n \in \mathbb{N}$ and any $x=(x_{ij}) \in M_{n}(\mathcal{M})$,
\begin{eqnarray}
\Vert \phi^{(n)} (x) \Vert \le \Vert \phi \Vert^{3} \Vert x \Vert. \label{4.2.1}
\end{eqnarray}

Fix $n \in \mathbb{N}$ and $x=(x_{ij}) \in M_{n}(\mathcal{M})$. We assume that $\Vert x \Vert =1$. Notice that $F = \{ x_{ij}: 1 \le i, j \le n \}$ is a finite subset of $\mathcal{M}$. By Lemma \ref{4.1}, there is a separable unital C$^*$-subalgebra $\mathcal{A}$ of $\mathcal{M}$ with Property c$^*$-$\Gamma$ such that $F \subseteq \mathcal{A}$. Let $\tilde{\phi}$ be the restriction of $\phi$ on $\mathcal{A}$. Then $\tilde{\phi}: \mathcal{A} \to B(H)$ is a bounded unital homomorphism. It was shown in  the proof of Theorem 5.3 in \cite{QS2} that
\begin{eqnarray}
\Vert \tilde{\phi} \Vert_{cb} \le \Vert \tilde{\phi} \Vert^{3}. \label{4.2.2}
\end{eqnarray}
Since $F \subseteq \mathcal{A}$, it follows from (\ref{4.2.2}) that
\begin{eqnarray*}
\Vert \phi^{(n)} (x) \Vert &=& \Vert \tilde{\phi}^{(n)} (x) \Vert \le \Vert \tilde{\phi} \Vert^{3} \le \Vert \phi \Vert^{3}.
\end{eqnarray*}

Therefore $d(\mathcal{M}) =3$ and the proof is completed.
\end{proof}

Based on Theorem \ref{4.2}, a slight  modification of the proof of Theorem 5.2 in \cite{QS2} gives the next corollary.

\begin{corollary} \label{4.3}
Let $\mathcal{M}$ be a von Neumann algebra with the type decomposition
$$\mathcal{M}= \mathcal{M}_{1} \oplus \mathcal{M}_{c_{1}} \oplus \mathcal{M}_{c_{\infty}} \oplus \mathcal{M}_{\infty},$$
where $\mathcal{M}_{1}$ is a type  I  von Neumann algebra,
$\mathcal{M}_{c_{1}}$ is a type  II$_{1}$ von Neumann algebra,
$\mathcal{M}_{c_{\infty}}$ is a type II$_{\infty}$ von Neumann
algebra and $\mathcal{M}_{\infty}$ is a type  III  von Neumann
algebra. Suppose  $\mathcal{M}_{c_{1}}$ is a type II$_{1}$
von Neumann algebra with  Property $\Gamma$. If $\phi$ is a bounded
unital representation of $\mathcal M$ on a Hilbert space $H$, which
is continuous from $\mathcal M$, with the topology $\sigma(\mathcal
M,\mathcal M_\sharp)$, to $B(H),$ with the topology
$\sigma(B(H),B(H)_\sharp)$, then $\phi$ is completely bounded and
$\Vert \phi \Vert_{cb} \leq \Vert \phi \Vert^{3}$.
\end{corollary}

Suppose $\mathcal{A}$ is a unital C$^*$-algebra. Let ${\mathcal I}$ be some index set and
$$l_{\infty} ({\mathcal I}, \mathcal{A}) = \{ (x_{i})_{i \in J}: \text{ for each } i \in {\mathcal I}, x_{i} \in \mathcal{A} \text{ and } \sup\limits_{i \in {\mathcal I}} \Vert x_{i} \Vert < \infty \}.$$
It was shown in \cite{Pi4} (Corollary 17) that if $\mathcal{M}$ is a type II$_{1}$ factor with Property $\Gamma$, then $d( l_{\infty} ({\mathcal I}, \mathcal{M}) ) \le 5$ for any index set $\mathcal I$. The next corollary gives an exact value of $d( l_{\infty} ({\mathcal I}, \mathcal{M}) )$.

\begin{corollary} \label{4.4}
If $\mathcal{M}$ is a type II$_{1}$ factor with Property $\Gamma$, then $d( l_{\infty} ({\mathcal I}, \mathcal{M}) ) = 3$ for any index set $\mathcal{I}$.
\end{corollary}
\begin{proof}
Assume that $\mathcal{M}$ is a type II$_{1}$ factor with Property $\Gamma$. By Theorem \ref{3.3}, for any index set ${\mathcal I}$, $ l_{\infty} ({\mathcal I}, \mathcal{M})$ is a type II$_1$ von Neumann algebra with Property $\Gamma$. Therefore
$$d( l_{\infty} ({\mathcal I}, \mathcal{M}) ) = 3.$$
\end{proof}

Let $C$ be the $CAR$-algebra $C=M_{2}(\mathbb{C}) \otimes M_{2}(\mathbb{C}) \otimes \dots$ (infinite C$^*$-tensor product of $2 \times 2$ matrix algebras). It was shown in \cite{Pi4} (Proposition 21) that, for any index set ${\mathcal I}$, $d(l_{\infty}({\mathcal I}, C)) \le 5$. The next corollary gives an exact value of $d(l_{\infty}({\mathcal I}, C))$.

\begin{corollary} \label{4.5}

Let $C=M_{2}(\mathbb{C}) \otimes M_{2}(\mathbb{C}) \otimes \dots$
(infinite C$^*$-tensor product of $2 \times 2$ matrix algebras).
Then, for any infinite index set ${\mathcal I}$, $d(l_{\infty}
({\mathcal I}, C) ) = 3$.

\end{corollary}
\begin{proof}
Denote by $\mathcal{A} = l_{\infty} ({\mathcal I}, C) = \sum\limits_{i \in \mathcal I} \oplus C_{i}$, where $C_{i}$ is a copy of $C$ for each $i \in \mathcal I$. Let $\mathcal R$ and $\mathcal R_i$ be the canonical hyperfinite II$_1$ factor generated by $C$ and $C_i$ respectively. Let $\tau_i$ be a trace on $\mathcal R_i$.  Let $\mathcal M=l_{\infty} ({\mathcal I}, \mathcal R)= \sum\limits_{i \in \mathcal I} \oplus \mathcal R_{i}$. We might assume that both $\mathcal M$ and $\mathcal A$ act naturally on the Hilbert space $\sum_{i\in\mathcal I} l^2(\mathcal R_i, \tau_i)$.
Denote by $p_{i}$ the projection in $\mathcal{A}$ such that $p_{i}\mathcal{A}=C_{i}$. It follows that $\sum\limits_{i \in \mathcal I}p_{i} =I$.

First we will prove the following two claims.

\vspace{0.2cm}

{Claim \ref{4.5}.1.}  For any $ x_1,\ldots, x_k $ in $\mathcal A$ and any $\epsilon>0$,   there exists  a system of matrix units  $\{E_{st}: 1\le s,t\le 2\}$  in $\mathcal A$ such that $E_{11}+E_{22}=I$ and $$\|x_jE_{ss}-E_{ss}x_j\| <\epsilon, \ \ \text{for all } 1\le s\le 2, \ 1\le j\le k.$$
\vspace{0.2cm}

\noindent Proof of Claim \ref{4.5}.1:   For each $i\in\mathcal I$, note that $p_i x_1,\ldots, p_i x_k$ are in a CAR algebra $C_i$. Hence there exists a system of matrix units $\{e_{st}^{(i)}: 1\le s,t\le 2\}$ in $C_i$ such that $e_{11}^{(i)}+e_{22}^{(i)}=p_i$ and $$\|x_je_{ss}^{(i)}-e_{ss}^{(i)}x_j\|<\epsilon/2, \ \ \text{for all } 1\le s\le 2, \ 1\le j\le k.$$
Let
$$
E_{st}=\sum_{i\in\mathcal I} e_{st}^{(i)}, \qquad \text {for all } 1\le s,t \le 2.
$$
 Then  $\{E_{st}: 1\le s,t\le 2\}$ is a system of matrix units in $\mathcal A$ such that $E_{11}+E_{22}=I$ and $$\|x_jE_{ss}-E_{ss}x_j\|=\sup_{i} \|x_je_{ss}^{(i)}-e_{ss}^{(i)}x_j\|<\epsilon, \ \ \text{for all } 1\le s\le 2, \ 1\le j\le k.$$ This finishes the proof of Claim \ref{4.5}.1.

\vspace{0.2cm}

{Claim \ref{4.5}.2.}  For any ${x_1,\ldots, x_k }$ in $\mathcal A$,   there exists a separable C$^*$-subalgebra $\mathcal B$ of $\mathcal A$ such that $\mathcal B$ is of Property c$^*$-$\Gamma$ and  all $x_1,\ldots x_k$ are in $\mathcal B$.

\vspace{0.2cm}

\noindent Proof of Claim \ref{4.5}.2:
Let $F_1=\{x_1,\ldots, x_k\}$. By Claim \ref{4.5}.1, there exists a system of matrix units $\{E_{st}^{(1)}: 1\le s,t\le 2\}$ in $\mathcal A$ such that $E_{11}^{(1)}+E_{22}^{(1)}=I$ and $$\|xE_{ss}^{(1)}-E_{ss}^{(1)}x\|<1. \ \ \text{for all } 1\le s\le 2, \ 1\le j\le k, \text { and } x\in F_1.$$
Let $F_2=F_1\cup \{E_{st}^{(1)}: 1\le s,t\le 2\}$.

Assume that $F_1\subseteq F_2\subseteq \cdots \subseteq F_m$ have
been constructed for some $m\ge 2$.  By Claim \ref{4.5}.1, we know there exists a  system
of matrix units $\{E_{st}^{(m)}: 1\le s,t\le 2\}$ in $\mathcal A$
such that $E_{11}^{(m)}+E_{22}^{(m)}=I$ and
$$\|xE_{ss}^{(m)}-E_{ss}^{(m)}x\|<\frac 1 m. \ \ \text{for all } 1\le s\le
2, \ 1\le j\le k, \text { and } x\in F_m.$$ Let $F_{m+1}=F_m\cup
\{E_{st}^{(m)}: 1\le s,t\le 2\}$.

Using similar  arguments as in Lemma \ref{4.1}, we are able to obtain an increasing sequence of subsets $\{F_m \}$ of $\mathcal A$ such that (a) the C$^*$-subalgebra $\mathcal B$ generated by  $\{F_m\}$ in $\mathcal A$ is of Property c$^*$-$\Gamma$; and (b)  all $x_1,\ldots x_k$ are in $\mathcal B$. This ends the proof of Claim \ref{4.5}.2.

\vspace{0.2cm}

\noindent(Continue the proof of Corollary:)
From Claim \ref{4.5}.2,  using similar arguments as in Theorem \ref{4.2}, we   conclude that $d(\mathcal A)\le 3$.

Next  we will show that $d(\mathcal A)\ge 3$. Since $\mathcal I$ is an infinite set, let $\mathcal I_0$ be a countable infinite subset of $\mathcal I$. Then $(\sum_{i\in \mathcal I\setminus \mathcal I_0} p_i)\mathcal A$ is a closed two sided ideal of $\mathcal A$. Moreover, $(\sum_{i\in   I_0} p_i)\mathcal A\cong \mathcal A/ (\sum_{i\in \mathcal I\setminus \mathcal I_0} p_i)\mathcal A .$ By Remark 6 in \cite{Pi1}, we know that $d(\mathcal A)\ge d((\sum_{i\in   I_0} p_i)\mathcal A)$. In order to show that $d(\mathcal A)\ge 3$, it suffices to show that $d((\sum_{i\in   I_0} p_i)\mathcal A)\ge 3$. By replacing $\mathcal I$ by $\mathcal I_0$, we can assume that $\mathcal I=\Bbb N$.

Let $\omega$ be a free ultra-filter of $\Bbb N$ and
$$
\mathcal J= \{(x_i)\in \mathcal M (=l_{\infty} ( \Bbb N, \mathcal R)= \sum\limits_{i \in \mathcal I} \oplus \mathcal R_{i}) : \lim_{i\rightarrow \omega} \tau_i(x_i^*x_i)=0\}
$$ be a closed two sided ideal of $\mathcal M$. By Theorem 7.1 in \cite{Sa2}, $\mathcal M/\mathcal J$ is a type II$_1$ factor. By Remark 12 in \cite{Pi4}, $d(\mathcal M/\mathcal J)\ge 3.$

Let $q:\mathcal M\rightarrow \mathcal M/\mathcal J$ be the quotient map. For any element $(x_i)\in \mathcal M$, by Kaplansky  Density Theorem, there exists an element  $(\tilde x_i)\in \mathcal A$ such that $q((\tilde x_i))=q(( x_i)).$ In other words, $q(\mathcal A)=\mathcal M/\mathcal J$. By Remark 6 in \cite{Pi1}, we get that $d(\mathcal A)\ge d(\mathcal M/\mathcal J)$. Combining with the result from the preceding paragraph, we conclude that $d(\mathcal A)\ge 3.$

 %Let $\mathcal Z$ be the center of $\mathcal M$. In fact $\mathcal Z$ is a von Neumann algebra generated by $\{p_i\}_{i\in \mathcal I}$ in $\mathcal %M$. Let $\tau$ be a center-valued trace   from $\mathcal M$ to $\mathcal Z$. Note $\mathcal Z\cong C(X)$ for some hyper-stonean space. Fix a $t$ in %$X$ and define
%$$
%\mathcal J= \{x\in\mathcal M : \tau(x^*x)(t)=0\}
%$$
%By Theorem ? in \cite{Sa2}, $\mathcal M/\mathcal J$ is a type II$_1$ factor. It is not hard to verify that $\mathcal A/\mathcal J\cong \mathcal %M/\mathcal J$. Since the similarity degree of a type II$_1$ factor is large than or equal to $3$, from Remark 6 in \cite{Pi1}, we get that %$d(\mathcal A)\ge 3.$

Therefore $d(l_{\infty} (\mathcal{I}, C) )=d(\mathcal A) = 3$, when
$\mathcal I$ is an infinite set.
\end{proof}

\end{document}